\documentclass{amsart}
\usepackage{amsmath,amssymb,amsthm}

\newtheorem{theorem}{Theorem}[section]
\newtheorem{lemma}[theorem]{Lemma}

\theoremstyle{definition}
\newtheorem{definition}[theorem]{Definition}

\theoremstyle{proposition}
\newtheorem{proposition}[theorem]{Proposition}

\theoremstyle{corollary}

\theoremstyle{remark}
\newtheorem{remark}[theorem]{Remark}

\numberwithin{equation}{section}
\begin{document}

\title[Global Regularity for 2-D
Boussinesq Equations]{Global Regularity of the Two-dimensional Boussinesq Equations Without
Diffusivity in Bounded Domains}

\author{Daoguo Zhou}

\address{School of Mathematics and Information Science, Henan Polytechnic University,
Jiaozuo, Henan 454000, China}

\email{daoguozhou@hpu.edu.cn}

\subjclass[2010]{35Q35, 76D05}

\keywords{Global regularity, Boussinesq equations, Zero diffusivity, Bounded
domain}
\begin{abstract}
We address the well-posedness for the two-dimensional Boussinesq equations
with zero diffusivity in bounded domains. We prove global in time
regularity for rough initial data: the initial velocity 
has  $\epsilon$ fractional derivatives in $L^{q}$ and the initial temperature is in $L^{q}$, for some
$q>2$ and $\epsilon>0$ arbitrarily small. 
\end{abstract}
\maketitle

\section{Introduction and Main Results}

In this paper, we study the initial-boundary value problem for the
2D Boussinesq equations with zero thermal diffusivity on an open bounded
domain $\Omega\subset\mathbb{R}^{2}$ with smooth boundary $\partial\Omega$.
The corresponding equations reads
\begin{equation}
\begin{cases}
\partial_{t}u+u\cdot\nabla u-\nu\Delta u+\nabla p=\theta e_{2},\\
\partial_{t}\theta+u\cdot\nabla\theta=0,\\
\nabla\cdot u=0,
\end{cases}
\end{equation}
where $u=(u_{1},u_{2})$ is the velocity vector field, $p$ is the
pressure, $\theta$ is the temperature, $\nu>0$ is the constant viscosity,
and $e_{2}=(0,1)$. This system is supplemented by the following initial
and boundary conditions 
\begin{equation}
\begin{cases}
(u,\theta)(x,0)=(u_{0},\theta_{0})(x),\;x\in\Omega,\\
u(x,t)|_{\partial\Omega}=0.
\end{cases}
\end{equation}
Here, we have imposed the mostly used no-slip conditions on the velocity,
which assume that fluid particles are adherent to the boundary due
to the positive viscosity.

The Boussinesq equations play an important role in modeling large scale atmospheric
and oceanic flows \cite{Ma}, \cite{Pe}. In addition, the Boussinesq
equations is closely related to Rayleigh-Benard convection \cite{Pe}.
From the mathematical view, the 2D Boussinesq equations serve as a
simplified model of the 3D Euler and Navier-Stokes equations. In fact,
we get the 2D Boussinesq equations when we analyze 3D axisymmetric
swirling fluid in the Navier-Stokes framework. Better understanding of
the 2D Boussinesq equations will undoubtedly shed light on the understanding
of 3D flows \cite{Ma2}.

Recently, the well-posedness of the 2D Boussinesq equations
has attracted attention of many mathematicians, see \cite{Abidi}-\cite{Chae},
\cite{DP}-\cite{Hu2}, \cite{Lai}, \cite{Li}, \cite{Titi}, \cite{Miao}, \cite{Sun},
\cite{Wang}. In particular, when $\Omega=\mathbb{R}^{2}$, the Cauchy
problem of (1.1) has been well studied. Hou and Li \cite{Hou} and
Chae \cite{Chae} showed the global in time regularity for $(u_{0},\theta_{0})\in H^{3}(\mathbb{R}^{2})\times H^{2}(\mathbb{R}^{2})$.
Kukavica, Wang, Ziane \cite{Ku} obtained  the global regularity for $(u_{0},\theta_{0})\in W^{1+s,q}(\mathbb{R}^{2})\times W^{s,q}(\mathbb{R}^{2})$ for $s\in (0,1)$, $q\in[2,\infty)$ and $sq>2$. They also pointed out that the restriction $sq>2$ can be removed provided that the initial data have compact support or $\Omega=\mathbb{T}^2$.
Abidi and Hmidi \cite{Abidi} proved the global existence for $(u_{0},\theta_{0})\in L^{2}(\mathbb{R}^{2})\cap B_{\infty,1}^{-1}(\mathbb{R}^{2})\times L^{2}(\mathbb{R}^{2})$.
Danchin and Paicu \cite{DP} proved the uniqueness of weak solution
for $(u_{0},\theta_{0})\in L^{2}(\mathbb{R}^{2})\times L^{2}(\mathbb{R}^{2})$.

In real world applications, fluids often move in bounded domains,
where new phenomena such as the creation of vorticity on the boundary
appears. In such case, the boundary effect requires a careful analysis.
The initial-boundary value problem of (1.1)-(1.2) was first studied
by Lai, Pan, and Zhao \cite{Lai}, who showed the global regularity
for $(u_{0},\theta_{0})\in H^{3}(\Omega)\times H^{2}(\Omega)$. Later,
Hu, Kukavica and Ziane \cite{Hu1} proved the global existence for
initial data $(u_{0},\theta_{0})\in H^{2}(\Omega)\times H^{1}(\Omega)$.
Recently, He \cite{He} established the uniqueness of weak solution
in the energy space $L^{2}(\Omega)\times L^{2}(\Omega)$.

In the current paper, we study further the initial-boundary value
problem of system (1.1)-(1.2). We improve the previous results to the case of rough
initial data $(u_{0},\theta_{0})\in D_{A_{q}}^{1-\frac{1}{p},p}(\Omega)\times W^{s,q}(\Omega)$.
Here, $D_{A_{q}}^{1-\frac{1}{p},p}$ denotes some fractional domain
of the Stokes operator whose elements have $2-\frac{2}{p}$ derivatives
in $L^{q}$ (see Section 2 for the definition).

Before stating our main results, we define the function spaces in
which existence is going to be shown.
 \begin{definition} 
 	For all
$T>0$, $s\geq 0$ and $1<p,q<\infty$, we denote $M_{T}^{p,q,s}$ the set of
triples $(u,p,\theta)$ such that 
\begin{gather*}
u\in C([0,T];D_{A_{q}}^{1-\frac{1}{p},p})\cap L^{p}(0,T;W^{2,q}\cap W_{0}^{1,q})\,,\partial_{t}u\in L^{p}(0,T;L^{q}),\\
p\in L^{p}(0,T;W^{1,q})\text{ and }\int_{\Omega}pdx=0,\\
\theta\in C([0,T];W^{s,q}),\,\partial_{t}\theta \in L^{p}(0,T;W^{-1,q}).
\end{gather*}
The corresponding norm is denoted by $\|\cdot\|_{M_{T}^{p,q,s}}$.
 \end{definition}
Our main results read as follows. 
\begin{theorem}\label{mainthm}
Let $\Omega$ be a bounded domain in $\mathbb{R}^{2}$ with $C^{2+\epsilon}$
boundary. Let  $p\in(1,\infty)$, $q\in(2,\infty)$ and
$s\in[0,1]$. Let $u_{0}\in D_{A_{q}}^{1-\frac{1}{p},p}$and
$\theta_{0}\in W^{s,q}$. Then system (1.1)-(1.2) has a unique global solution which
belongs to $M_{T}^{p,q,s}$ for all $T>0$. 
\end{theorem}

As a byproduct of the proof of Theorem \ref{mainthm}, we get the
global regularity for $(u_{0},\theta_{0})\in H^{1}(\Omega)\times H^{1}(\Omega)$.

\begin{proposition}\label{proposition} Let $\Omega$ be a bounded
domain in $\mathbb{R}^{2}$ with $C^{2+\epsilon}$ boundary. Assume
that $u_{0}\in H^{1}$, $\nabla\cdot u_{0}=0$ and $\theta_{0}\in H^{1}$.
Then there exists a unique global solution to system (1.1)-(1.2) such that
for all $T>0$ 
\[
u\in L^{\infty}(0,T;H^{1})\cap L^{2}(0,T;H^{2}),\;\nabla u\in L^{1}(0,T;L^{\infty}),
\]
\[
\theta\in L^{\infty}(0,T;H^{1}).
\]
\end{proposition}

\begin{remark}
While preparing the manuscript, the author becomes to know that Proposition
\ref{proposition} was obtained very recently by Ju \cite{Ju} independently.
However, our method is completely different from that of Ju, which
exploited Brezis-Gallouet type inequalities and spectral decomposition.
\end{remark}

\begin{remark}The proof of Theorem \ref{mainthm}
	is based on the maximal regularity of the Stokes operator and some
	interpolation inequalities. Our method is elementary and can be carried
	over to the whole space case $\mathbb{R}^{2}$ without difficulty.
\end{remark}

\begin{remark}
 Since the ``regularity index'' $2-2/p$ in Theorem \ref{mainthm} can be arbitrarily close to zero, our result improves
the previous works of Lai et al. \cite{Lai},  Hu et al. \cite{Hu1} and Kukavica  et al. \cite{Ku} by requiring much less regularity for the initial data. 
\end{remark}

The rest of this paper is organized as follows. In Section 2, we recall
the maximal regularity of Stokes equations as well as some elementary
inequalities. In Section 3, we present the detailed proofs of the
main results.

\section{preliminaries}

\noindent \textbf{Notations:}

(1) Let $\Omega$ be a bounded domain in $\mathbb{R}^{2}$. For $1<q<\infty$,
denote by $L_{\sigma}^{q}(\Omega)$ the completion in $L^{q}(\Omega)$
of the set of solenoidal vector-fields with coefficients in $C_{0}^{\infty}(\Omega)$.
If $k$ is an integer, we denote by $W^{k,q}(\Omega)$ the set of
$L^{q}(\Omega)$ functions whose derivatives up to order $k$ belong
to $L^{q}(\Omega)$. For $s\in (0,1)$, the Sobolev space $W^{s,q}(\Omega)$ is defined as
\[W^{s,q}(\Omega)=\left\{f\mid \|f\|_{W^{s,q}}=\left(\int_{\Omega}\int_{\Omega}\frac{|f(x)-f(y)|^q}{|x-y|^{sq+2}}dxdy\right)^{\frac{1}{q}}<\infty\right\}.\]

(2) For $\alpha\in(0,2)$ and $1<p,q<\infty$, denote by $B_{q,p}^{\alpha}$
the Besov space which is defined as the real interpolation space between
$L^{q}(\Omega)$ and $W^{m,q}(\Omega)$ ($m>\alpha$):
\[
B_{q,r}^{\alpha}(\Omega)=(L^{q}(\Omega),W^{m,q}(\Omega))_{\frac{\alpha}{m},r}.
\]
Denote by $\mathring{B}_{q,p}^{\alpha}$ the completion of $C_{0}^{\infty}(\Omega)$
in $B_{q,p}^{\alpha}$. See Adams and Fournier \cite{AD} for more about the Besov space.

(3) For $T>0$ and a function space $X$, denote by $L^{p}(0,T;X)$
the set of Bochner measurable $X$-valued time dependent functions
$f$ such that $t\to\|f\|_{X}$ belongs to $L^{p}(0,T)$.

First we give the definition of the fractional domains of the Stokes
operator in $L^{q}$. \begin{definition} For $\alpha\in(0,1)$ and
$s,q\in(1,\infty)$, we set 
\[
\|u\|_{D_{A_{q}}^{\alpha,s}}=\|u\|_{L^{q}}+\left(\int_{0}^{\infty}\|t^{1-\alpha}A_{q}e^{-tA_{q}}u\|_{L^{q}}^{s}\frac{dt}{t}\right)^{\frac{1}{s}},
\]
where $A_{q}=-\mathbb{P}\Delta$ has domain $D(A_{q})=W^{2,q}(\Omega)\cap W_{0}^{1,q}(\Omega)\cap L_{\sigma}^{q}(\Omega)$.
Here, $\mathbb{P}$ denotes the Leray projector. \end{definition}

Roughly, the vector fields of $D_{A_{q}}^{\alpha,s}$ have $2\alpha$ derivatives in $L^{q}$, are divergence-free, and
vanish on $\partial\Omega$. In fact, we have the following imbedding
(cf. Proposition 2.5 in Danchin \cite{DA}). \begin{lemma} \label{imbedding}
$\mathring{B}_{q,s}^{2\alpha}\cap L_{\sigma}^{q}\hookrightarrow D_{A_{q}}^{\alpha,s}\hookrightarrow B_{q,s}^{2\alpha}\cap L_{\sigma}^{q}$,
for $\alpha\in(0,1)$, $1<q,s<\infty$. Moreover, if $2\alpha\leq\frac{1}{q}$,
then the three spaces are the same (with equivalent norms). \end{lemma}
We need the well-known Sobolev embedding, Ladyzhenskaya inequality
and Gagliardo-Nirenberg inequality (see Adams and Fournier \cite{AD},
Ladyzhenskaya \cite{Ladyzhen} and Nirenberg \cite{Niren}).

\begin{lemma}\label{embed} Let $\Omega\in\mathbb{R}^{2}$ be a
bounded domain with $C^{2}$ boundary. Then the following embeddings
and inequalities hold: 
\begin{enumerate}
\item $H^{1}(\Omega)\hookrightarrow L^{q}(\Omega)$, for all $q\in(1,\infty)$. 
\item $\|f\|_{L^{4}}\leq\sqrt{2}\|f\|_{L^{2}}^{\frac{1}{2}}\|\nabla f\|_{L^{2}}^{\frac{1}{2}}$,
for $f\in H_{0}^{1}(\Omega)$. 
\item $\|\nabla u\|_{L^{\infty}}\leq C\|\nabla^{2}u\|_{L^{q}}^{\alpha}\|u\|_{L^{q}}^{1-\alpha}+C\|u\|_{L^{q}}$,
for all $u\in W^{2,q}(\Omega)$, with $q\in(2,\infty)$, $\alpha=\frac{1}{2}+\frac{1}{q}$
and $C$ is a constant depending on $q,\Omega$. 
\end{enumerate}
\end{lemma} 
We also need the following interpolation inequality (cf. Lemma 4.1 in Danchin
~\cite{DA}). 
\begin{lemma}\label{inter} Let
$1<p,q<\infty$ satisfy $0<\frac{1}{2}-\frac{1}{q}<\frac{1}{p}$.
The following inequality holds true:
\[
\|\nabla u\|_{L^{p}(0,T;L^{\infty})}\leq CT^{\frac{1}{2}-\frac{1}{q}}\|f\|_{L^{\infty}(0,T;D_{A_{q}}^{1-\frac{1}{p}})}^{1-\theta}\|f\|_{L^{p}(0,T;W^{2,q})}^{\theta},
\]
for $C=(p,q,\Omega)$ and $\frac{1-\theta}{p}=\frac{1}{2}-\frac{1}{q}$.
\end{lemma}

Now we recall the following standard result for linear transport equations.
\begin{lemma}\label{tran}
Let $\Omega$ be a Lipschitz domain of $\mathbb{R}^2$ and $u\in L^{1}(0,T;W^{1,\infty})$
such that $\textrm{div }u=0$ and $u\cdot n=0$ on $\partial\Omega$.
Let $a_{0}\in W^{s,q}$ with $q\in[1,\infty)$ and $s\in[0,1]$.
Then the equation 
\begin{equation}\label{transe}
\begin{cases}
\partial_{t}a+u\cdot\nabla a=0,\\
a(x,t)|_{t=0}=a_{0},
\end{cases}
\end{equation}
has a unique solution in $C([0,T];W^{s,q})$. Moreover,
the following estimate holds true for all $t\in[0,T]$ 
\[
\|a(t)\|_{W^{s,q}}\leq\|a_{0}\|_{W^{s,q}}e^{C\int_{0}^{t}\|\nabla v(\tau)\|_{L^{\infty}}d\tau},
\]
with $C=C(s,q)$.
If in addition $a$ belongs to $L^{p}$ for some $p\in[1,\infty]$
then for all $t\in[0,T]$ 
\[
\|a(t)\|_{L^{p}}=\|a_{0}\|_{L^{p}}.
\]
\end{lemma}
The result for $s=0$ and $s=1$ is well-known (cf. Proposition 3.1 in Danchin ~\cite{DA})). The case $s\in(0,1)$ 
seems to be folklore, but I can not locate the proof. Here we provide a sketched proof.  
\begin{proof}
We only establish a priori estimates and refer the reader to Desjardins \cite{De} for the existence and uniqueness parts.

Denote by $\psi_t(x)$ the flow of $u$, which is defined by 
\[\partial_t \psi_t(x)=u(\psi_t(x),t),\quad \psi_t(x)|_{t=0}=x\in \Omega,\]
It follows that \eqref{transe}  has the formal solution
\[a(x,t)=a_0(\psi_t^{-1}(x)).\]
From the assumptions $\text{div } u=0$ and $u\in L^{1}(0,T;W^{1,\infty})$, we obtain
\begin{equation}\label{par}
|\text{det }\psi_t(x)|=1,\text{ and } |x-y|\leq |\psi_t(x)-\psi_t(y)|e^{\int_0^t\|\nabla u(\tau)\|_{L^\infty}{d\tau}}.
\end{equation}
See Chapter 4 in Majda \cite{Ma} for more details.

Using the definition of fractional Sobolev space and \eqref{par}, we can compute $\|a(t)\|_{W^{s,q}}$ as follows:
\begin{align*}
\|a(t)\|_{W^{s,q}}&=\left(\int_{\Omega}\int_{\Omega}\frac{|a_0(\psi_t^{-1}(x))-a_0(\psi_t^{-1}(y))|^q}{|x-y|^{sq+2}}dxdy\right)^{\frac{1}{q}}\\
&=\left(\int_{\Omega}\int_{\Omega}\frac{|a_0(u)-a_0(v)|^q}{|\psi_t(u)-\psi_t(v)|^{sq+2}}|\text{det}\nabla \psi_t(u)||\text{det}\psi_t(v)|dudv\right)^{\frac{1}{q}}\\
&=\left(\int_{\Omega}\int_{\Omega}\frac{|a_0(u)-a_0(v)|^q}{|u-v|^{sq+2}}\left(\frac{|u-v|}{|\psi_t(u)-\psi_t(v)|}\right)^{sq+2}dudv\right)^{\frac{1}{q}}\\
&\leq \|a_0\|_{W^{s,q}}e^{(s+\frac{2}{q})\int_0^t\|\nabla u(\tau)\|_{L^\infty}{d\tau}}.
\end{align*}
This completes the proof of Lemma \ref{tran}.
\end{proof}

We conclude this section by recalling the maximal regularity of the
Stokes equations (cf. Theorem 3.2 in Danchin ~\cite{DA} or Theorem
1.1 in Solonnikov \cite{Sol}), which will be used in the proof of
Theorem \ref{mainthm}. \begin{lemma}\label{MR} Let $\Omega$ be a bounded
domain with a $C^{2+\epsilon}$ boundary in $\mathbb{R}^{2}$ and
$1<p,q<\infty$. Assume that $u_{0}\in D_{A_{q}}^{1-\frac{1}{p},p}$,
$f\in L^{p}(0,\infty;L^{q})$. Then the system 
\[
\begin{cases}
\partial_{t}u-\nu\Delta u+\nabla p=f,\\
\nabla\cdot u=0,\\
u(x,t)|_{\partial\Omega}=0,\;u(x,t)|_{t=0}=u_{0},
\end{cases}
\]
has a unique solution $(u,p)$ satisfying the following inequality
for all $T>0$: 
\begin{align*}
 & \|u\|_{L^\infty(0,T;D_{A_{q}}^{1-\frac{1}{p},p})}+\|u\|_{L^{p}(0,T;W^{2,q})}+\|\partial_{t}u\|_{L^{p}(0,T;L^{q})}+\|p\|_{L^{p}(0,T;W^{1,q})}\\
 & \leq C\left(\|u_{0}\|_{D_{A_{q}}^{1-\frac{1}{p},p}}+\|f\|_{L^{p}(0,T;L^{q})}\right),
\end{align*}
with $C=C(p,q,\nu,\Omega)$. \end{lemma}

\section{proof of main results}

In this section, we prove Theorem \ref{mainthm} and Proposition \ref{proposition}.
To do so, we make two preparations. The first is a local existence
result for system (1.1)-(1.2). \begin{lemma}\label{local} Let the conditions
in Theorem \ref{mainthm} hold. Then there exists a $T_{0}=T_{0}(\|u_{0}\|_{D_{A_{q}}^{1-\frac{1}{p},p}},\|\theta_{0}\|_{L^{q}})>0$
such that  system (1.1)-(1.2) has a unique solution in $M_{T_{0}}^{p,q,s}$.
\end{lemma}

\begin{proof} The proof consists of several steps, including constructing
the approximate solutions, obtaining the uniform local
in time estimates, showing the convergence, and proving the uniqueness.

\textbf{First step: Construction of approximate solutions.} We initialize
the construction of approximate solutions by smoothing out the initial
data $(u_{0},\theta_{0})$ and get a sequence of smooth initial data
$(u_{0}^{n},\theta_{0}^{n})_{n\in\mathbb{N}}$ which is bounded in
$D_{A_{q}}^{1-\frac{1}{p},p}\times L^q$. In addition, these smooth
data belong to the Sobolev space $H^3$. Hence, applying the result
of Lai, Pan, and Zhao \cite{Lai} provides us a sequence of smooth
global solutions $(u^{n},p^{n},\theta^{n})_{n\in\mathbb{N}}$, which  satisfy that  $(u^{n},\theta^{n})\in C([0,\infty);H^{3})\cap C^{1}([0,\infty);H^{2})$ and $p^{n}\in C([0,\infty);H^{3})$.

\textbf{Second step: Uniform estimate for some small fixed time $T_0$.}
We aim at finding a positive time $T_0$ independent of $n$ for which
$(u^{n},p^{n},\theta^{n})_{n\in\mathbb{N}}$ is uniformly bounded
in the space $M_{T_{0}}^{p,q,s}$.

Applying Lemma \ref{tran} to the temperature equation, we find that
for all $t\geq0$ and $s\in[0,1]$
\begin{equation}
\|\theta^{n}\|_{L^{\infty}(0,t;L^{q})}\leq\|\theta_{0}\|_{L^{q}},\label{tem}
\end{equation}
and 
\begin{equation}
\|\theta^{n}\|_{L^{\infty}(0,t;W^{s,q})}\leq\|\theta_{0}\|_{W^{s,q}}e^{C\int_{0}^{t}\|\nabla u^{n}(\tau)\|_{L^{\infty}}d\tau}.\label{temgra}
\end{equation}

 Considering the velocity equation, we obtain
\begin{equation}
\begin{aligned} & \|u^{n}\|_{L^{\infty}(0,t;D_{A_{q}}^{1-\frac{1}{p},p})}+\|\partial_{t}u^{n}\|_{L^{p}(0,t;L^{q})}+\|u^{n}\|_{L^{p}(0,t;W^{2,q})}+\|p^{n}\|_{L^{p}(0,t;W^{1,q})}\\
 & \leq C\left(\|u_{0}\|_{D_{A_{q}}^{1-\frac{1}{p},p}}+\|u^{n}\cdot\nabla u^{n}\|_{L^{p}(0,t;L^{q})}+\|\theta^{n}\|_{L^{p}(0,t;L^{q})}\right)\\
 & \leq C\left(\|u_{0}\|_{D_{A_{q}}^{1-\frac{1}{p},p}}+\|u^{n}\|_{L^{\infty}(0,t;L^{q})}\|\nabla u^{n}\|_{L^{p}(0,t;L^{\infty})}+\|\theta^{n}\|_{L^{p}(0,t;L^{q})}\right).
\end{aligned}
\label{localu}
\end{equation}
If $\frac{2}{p}+\frac{2}{q}>1$, Lemma \ref{inter} yields for $\theta=1-\frac{p}{2}(1-\frac{2}{q})$
\begin{equation}
\|\nabla u^{n}\|_{L^{p}(0,t;L^{\infty})}\leq Ct^{\frac{1}{2}-\frac{1}{q}}\|u^{n}\|_{L^{p}(0,t;W^{2,q})}^{\theta}\|u^{n}\|_{L^{\infty}(0,t;D_{A_{q}}^{1-\frac{1}{p},p})}^{1-\theta}.\label{inter1}
\end{equation}
If $\frac{2}{p}+\frac{2}{q}<1$, we have $D_{A_{q}}^{1-\frac{1}{p},p}\hookrightarrow W^{1,\infty}$ so that
\begin{equation}
\|\nabla u^{n}\|_{L^{p}(0,t;L^{\infty})}\leq Ct^{\frac{1}{p}}\|u^{n}\|_{L^{\infty}(0,t;D_{A_{q}}^{1-\frac{1}{p},p})}.\label{inter2}
\end{equation}
If $\frac{2}{p}+\frac{2}{q}=1$, applying H\"{o}lder's inequality, we
arrive at the following inequality, 
\begin{equation}
\|u^{n}\cdot\nabla u^{n}\|_{L^{p}(0,t;L^{q})}\leq\|u^{n}\|_{L^{\infty}(0,t;L^{q+})}\|\nabla u^{n}\|_{L^{p}(0,t;L^{\frac{qq^{+}}{q^{+}-q}})},\label{inter3}
\end{equation}
where $q^{+}$ is slightly bigger than $q$. Noticing that $D_{A_{q}}^{1-\frac{1}{p},p}\hookrightarrow W^{1,\frac{qq^{+}}{q^{+}-q}}$ and
$D_{A_{q}}^{1-\frac{1}{p},p}\hookrightarrow L^{q^{+}}$, we eventually
get 
\begin{equation}
\|u^{n}\cdot\nabla u^{n}\|_{L^{p}(0,t;L^{q})}\leq Ct^{\frac{1}{p}}\|u^{n}\|_{L^{\infty}(0,t;D_{A_{q}}^{1-\frac{1}{p},p})}^{2}.\label{inter4}
\end{equation}
On the other hand, combining H\"{o}lder's inequality and \eqref{tem},
we have
\begin{equation}
\|\theta^{n}\|_{L^{p}(0,t;L^{q})}\leq t^{\frac{1}{p}}\|\theta^{n}\|_{L^{\infty}(0,t;L^{q})}\leq t^{\frac{1}{p}}\|\theta_{0}\|_{L^{q}}.\label{inter5}
\end{equation}

Define 
\[
U^{n}(t)=\|u^{n}\|_{L^{\infty}(0,t;D_{A_{q}}^{1-\frac{1}{p},p})}+\|\partial_{t}u^{n}\|_{L^{p}(0,t;L^{q})}+\|u^{n}\|_{L^{p}(0,t;W^{2,q})},
\]
and
\[
\|U^{0}\|=\|u_{0}\|_{D_{A_{q}}^{1-\frac{1}{p},p}}+\|\theta_{0}\|_{L^{q}}.
\]
Inserting \eqref{inter1}-\eqref{inter5} into \eqref{localu}, we
deduce that 
\begin{equation}
U^{n}(t)\leq C\left(\|u_{0}\|_{D_{A_{q}}^{1-\frac{1}{p},p}}+\max(t^{\frac{1}{2}-\frac{1}{q}},t^{\frac{1}{p}})(U^{n}(t))^{2}+t^{\frac{1}{p}}\|\theta_{0}\|_{L^{q}}\right).\label{ineU}
\end{equation}
Set
\[
T_{0}=\min\left\{ 1,\left(\frac{1}{4CU_{0}}\right)^{p},\left(\frac{1}{4CU_{0}}\right)^{\frac{2q}{q-2}}\right\}.
\]
Direct computations show that for $t\in[0,T_{0}]$ 
\begin{equation}\label{localR}
U^{n}(t)\leq I_1(t)=\frac{1-\sqrt{1-4C\max(t^{\frac{1}{2}-\frac{1}{q}},t^{\frac{1}{p}})U_{0}}}{2\max(t^{\frac{1}{2}-\frac{1}{q}},t^{\frac{1}{p}})}\leq2CU_{0},
\end{equation}
or
\begin{equation}\label{localWR}
U^{n}(t)\geq I_2(t)=\frac{1+\sqrt{1-4C\max(t^{\frac{1}{2}-\frac{1}{q}},t^{\frac{1}{p}})U_0}}{2\max(t^{\frac{1}{2}-\frac{1}{q}},t^{\frac{1}{p}})}\geq\frac{1}{2\max(t^{\frac{1}{2}-\frac{1}{q}},t^{\frac{1}{p}})}.
\end{equation}
We show that \eqref{localR} holds for $t\in [0,T_0]$. Since  $u^{n}\in C([0,\infty);H^{3})$ and $\lim_{t\to 0}I_2(t)=\infty$, there exists some time $T_1>0$ such that \eqref{localR} holds for $t\in [0,T_1]$. By contradiction, suppose that \eqref{localR} does not hold for all $t\in [0,T_0]$, then there exists a first time  $T_2>0$ such that \eqref{localWR} holds. It follows that $\lim_{t\to T_2-}U^n(t)\leq I_1(T_2)$ and $U^n(T_2)\geq I_2(T_2)$, which contradicts the fact $u^{n}\in C([0,\infty);H^{3})$.  Hence we have for all $t\in[0,T_{0}]$
\begin{equation}
U^{n}(t)\leq2CU_{0}.
\end{equation}

Coming back to \eqref{temgra}, noting that \[\int_0^t\|\nabla u^{n}(\tau)\|_{L^{\infty}}d\tau \leq \int_0^t\|u^n\|_{W^{2,q}}d\tau\leq t^{1-\frac{1}{p}}\|u^n\|_{L^p(0,T;W^{2,q})},\] 
we derive that
\begin{equation}\label{Bout}
\theta\in C([0,T];W^{s,q}),\quad \partial_{t}\theta^{n}\in L^{p}(0,T;W^{-1,q}).
\end{equation}

\textbf{Third step: Passing to the limit.} 
Since $(u^{n},p^{n},\theta^{n})_{n\in\mathbb{N}}$ is uniformly bounded
in the space $M_{T_{0}}^{p,q,s}$, applying Aubin-Lions lemma yields the solution to system (1.1)-(1.2) which belongs to $ M_{T_{0}}^{p,q,s}$.

\textbf{Fifth step: Uniqueness.}
The uniqueness is implied by the result in He \cite{He}, which says that the energy weak solution to system (1.1)-(1.2) is unique. This completes the proof of Lemma \ref{mainlemma}.
\end{proof} 

The following lemma is the main ingredient of the proof of Theorem \ref{mainthm}.
\begin{lemma}\label{mainlemma} Let $\Omega$
be a bounded domain in $\mathbb{R}^{2}$ with $C^{2+\epsilon}$ boundary
for some $\epsilon>0$. Let $p\in(1,\infty)$, $q\in(2,\infty)$ and
$s\in[0,1]$. Suppose that $u_{0}\in D_{A_{q}}^{1-\frac{1}{p},p}\cap H^{1}$
and $\theta_{0}\in W^{s,q}$. Then system (1.1)-(1.2) has a unique global solution $(u,p,\theta)$
which belongs to $M_{T}^{p,q,s}$ for all $T>0$. Furthermore, we
have 
\[
u\in L^{\infty}(0,T;H^{1})\cap L^{2}(0,T;H^{2}).
\]
\end{lemma} We divide the proof of Lemma \ref{mainlemma} into three
steps. First, we recall some elementary energy estimates. Next, we
derive global $H^{1}$ estimate for the velocity. Finally, we use the maximal
regularity of the Stokes operator to improve the regularity for both
the velocity and the temperature.

\begin{proof} \textbf{Step 1 Energy Estimates.}

Let $T>0$ be any fixed given time. Reasoning as in Lemma \ref{local},
we get from the temperature equation for all $r\in[1,\infty)$
\begin{equation}
\theta\in L^{\infty}(0,T;L^{r}).\label{boundtheta}
\end{equation}
The basic energy estimate for the velocity equation yields that 
\[
\frac{1}{2}\frac{d}{dt}\|u\|_{L^{2}}^{2}+\nu\|\nabla u\|_{L^{2}}^{2}\leq\|u\|_{L^{2}}\|\theta\|_{L^{2}}.
\]
Applying Gronwall's inequality, we have 
\begin{equation}
u\in L^{\infty}(0,T;L^{2})\cap L^{2}(0,T;H^{1}).\label{energyu}
\end{equation}

\textbf{Step 2 $H^{1}$ Estimate for the Velocity.}

Taking $L^{2}$-inner product of the velocity equation with $-\mathbb{P}\Delta u$,
where $\mathbb{P}$ is the Leray projector. we deduce that 
\begin{equation}\label{gradient}
\frac{1}{2}\frac{d}{dt}\|\nabla u\|_{L^{2}}^{2}+\nu\|\mathbb{P}\Delta u\|_{L^{2}}^{2}=\int_{\Omega}u\cdot\nabla u\mathbb{P}\Delta udx-\int_{\Omega}\theta e_{2}\mathbb{P}\Delta udx.
\end{equation}
We now estimate the right hand side of \eqref{gradient}. For the
first term, using H\"{o}lder's inequality, Gagliardo-Nirenberg's inequality
and Young's inequality, we get 
\begin{equation}\label{nonline}
\begin{aligned}\left|\int_{\Omega}u\cdot\nabla u\mathbb{P}\Delta udx\right| & \leq C\|u\|_{L^{4}}\|\nabla u\|_{L^{4}}\|\mathbb{P}\Delta u\|_{L^{2}}\\
 & \leq C\|u\|_{L^{2}}^{\frac{1}{2}}\|\nabla u\|_{L^{2}}\|\mathbb{P}\Delta u\|_{L^{2}}^{\frac{3}{2}}\\
 & \leq C\|u\|_{L^{2}}^{2}\|\nabla u\|_{L^{2}}^{4}+\frac{1}{4}\nu\|\mathbb{P}\Delta u\|_{L^{2}}^{2}.
\end{aligned}
\end{equation}
For the second term, it follows from the Cauchy-Schwarz inequality
that 
\begin{equation}\label{inetheta}
\left|\int_{\Omega}\theta e_{2}\mathbb{P}\Delta udx\right|\leq C\|\theta\|_{2}^{2}+\frac{1}{4}\nu\|\mathbb{P}\Delta u\|_{2}^{2}
\end{equation}
Substituting \eqref{nonline} and \eqref{inetheta} into \eqref{gradient}, we find that 
\[
\frac{d}{dt}\|\nabla u\|_{L^{2}}^{2}+\nu\|\mathbb{P}\Delta u\|_{L^{2}}^{2}\leq C\left(\|u\|_{L^{2}}^{2}\|\nabla u\|_{L^{2}}^{2}\right)\|\nabla u\|_{L^{2}}^{2}+\|\theta\|_{L^{2}}^{2}.
\]
Then, from  \eqref{boundtheta}, \eqref{energyu} and Gronwall's
inequality, we obtain 
\begin{equation}
u\in L^{\infty}(0,T;H^{1})\cap L^{2}(0,T;H^{2}),
\end{equation}
which, by the Sobolev embedding
(Lemma \ref{embed}), implies that for all $q\in(2,\infty)$ 
\begin{equation}
u\in L^{\infty}(0,T;L^{q}).\label{boundu}
\end{equation}

\textbf{Step 3 Bootstrap Argument.}

We derive $W^{2,p}$ estimate for the velocity by the maximal regularity of the Stokes operator. To this end, we rewrite
the velocity equation as follows, 
\begin{align*}
\partial_{t}u-\nu\Delta u+\nabla p & =-u\cdot\nabla u+\theta e_{2},\\
\nabla\cdot u & =0.
\end{align*}
Using Lemma \ref{MR}, we see that for $p\in(1,\infty)$, $q\in(2,\infty)$,
\begin{equation}
\begin{aligned} & \|u\|_{L^{\infty}(0,T;D_{A_{q}}^{1-\frac{1}{p},p})}+\|\partial_{t}u\|_{L^{p}(0,T;L^{q})}+\|u\|_{L^{p}(0,T;W^{2,q})}+\|p\|_{L^{p}(0,T;W^{1,q})}\\
 & \leq C\left(\|u_{0}\|_{D_{A_{q}}^{1-\frac{1}{p},p}}+\|u\cdot\nabla u\|_{L^{p}(0,T;L^{q})}+\|\theta\|_{L^{p}(0,T;L^{q})}\right).
\end{aligned}
\label{mainine}
\end{equation}
We now estimate the term $\|u\cdot\nabla u\|_{L^{p}(0,T;L^{q})}$.
Applying the interpolation inequality in Lemma 2.2, H\"{o}lder's inequality
and Young's inequality, we find for any $\epsilon>0$ 
\begin{align*}
\|u\cdot\nabla u\|_{L^{p}(0,T;L^{q})} & \leq\|\nabla u\|_{L^{p}(0,T;L^{\infty})}\|u\|_{L^{\infty}(0,T;L^{q})}\\
 & \leq C\left(\|\nabla^{2}u\|_{L^{p}(0,T;L^{q})}^{\alpha}\|u\|_{L^{p}(0,T;L^{q})}^{1-\alpha}+\|u\|_{L^{p}(0,T;L^{q})}\right)\|u\|_{L^{\infty}(0,T;L^{q})}\\
 & \leq C\left(\epsilon\|\nabla^{2}u\|_{L^{p}(0,T;L^{q})}+C(\epsilon)\|u\|_{L^{p}(0,T;L^{q})}\right)\|u\|_{L^{\infty}(0,T;L^{q})}\\
 & \leq C\epsilon\|u\|_{L^{\infty}(0,T;L^{q})}\|\nabla^{2}u\|_{L^{p}(0,T;L^{q})}+C(\epsilon)T^{\frac{1}{p}}\|u\|_{L^{\infty}(0,T;L^{q})}^{2}.
\end{align*}
Choosing $\epsilon$ small such that $C\epsilon\|u\|_{L^{\infty}(0,T;L^{q})}\leq\frac{1}{2C}$,
we get 
\begin{equation}
\|u\cdot\nabla u\|_{L^{p}(0,T;L^{q})}\leq\frac{1}{2C}\|\nabla^{2}u\|_{L^{p}(0,T;L^{q})}+CT^{\frac{1}{p}}\|u\|_{L^{\infty}(0,T;L^{q})}^{2}.\label{Boundnonl}
\end{equation}
Substituting  \eqref{Boundnonl} into \eqref{mainine}, together with
\eqref{boundtheta} and  \eqref{boundu}, we deduce that 
\begin{align*}
 & {}\|u\|_{L^{\infty}(0,T;D_{A_{q}}^{1-\frac{1}{p},p})}+\|\partial_{t}u\|_{L^{p}(0,T;L^{q})}+\|u\|_{L^{p}(0,T;W^{2,q})}+\|p\|_{L^{p}(0,T;W^{1,q})}\\
 & \leq C(\|u_{0}\|_{D_{A_{q}}^{1-\frac{1}{p},p}},\|u_{0}\|_{H^{1}},\|\theta_{0}\|_{L^{q}},T).
\end{align*}

Finally, reasoning similarly as in Lemma \ref{local}, we obtain estimate for the temperature in $L^{\infty}(0,T;W^{s,q})$. This completes the
proof of Lemma \ref{mainlemma}. \end{proof}

With Lemma \ref{local} and \ref{mainlemma} at hand, we are at a position to
prove Theorem \ref{mainthm}. If the initial velocity is
smooth such that $u_{0}\in H^{1}$, then the proof is easy due to the global bound $u\in L^{\infty}(0,T;H^{1})\cap L^{2}(0,T;H^{2})$
for all $T>0$ (see Lemma \ref{mainlemma}). However, if the initial velocity is rough such that $u_{0}\notin H^{1}$, then the global $H^{1}$ bound for the velocity
is absent, which makes it difficult to improve the regularity for the velocity by bootstrap argument. To solve this issue, we shall exploit the continuation argument due to Danchin (cf. Section 7 in \cite{DA}).

\begin{proof}[Proof
of Theorem \ref{mainthm}] We treat two cases $p\geq2$ and
$1<p<2$ differently.

\textbf{(1) The Case of Smooth Data}  $p\geq2$. Combining the imbedding
$D_{A_{q}}^{1-\frac{1}{p},p}\hookrightarrow H^{1}$ (see Lemma \ref{imbedding}) and Lemma \ref{mainlemma} yields the result.

\textbf{(2) The Case of Rough Data} $1<p<2$. First, Lemma \ref{local}
gives us a local smooth solution $(u,p,\theta)$ with the initial
data $(u_{0},\theta_{0})$. Let $T^{*}\in(0,\infty)$ be the existence
time such that $(u,p,\theta)$ belongs to $M_{T^{*}}^{p,q,s}$. Then
we shall prove that $T^{*}$ can be arbitrarily large by adapting
the method of Danchin \cite{DA}.

Since $u\in L^{p}(0,T^{*};W^{2,q})$ and $\theta\in L^{\infty}(0,T^{*};W^{s,q})$,
there exists some $t_{0}\in(0,T^{*})$ such that $u(t_{0})\in W^{2,q}\cap L_{\sigma}^{q}$ and
$\theta(t_{0})\in W^{s,q}$. Noticing that $W^{2,q}\cap L_{\sigma}^{q}\hookrightarrow D_{A_{q}}^{1-\frac{1}{p},p}$
(see Lemma \ref{imbedding}), we obtain $u(t_{0})\in D_{A_{q}}^{1-\frac{1}{p},p}\cap H^{1}$.
Due to Lemma \ref{mainlemma}, we can find a unique global smooth 
solution $(\tilde{u},\tilde{\theta},\tilde{p})$ to system (1.1)-(1.2) with initial
data $(u(t_{0}),\theta(t_{0}))$.

On the other hand, because the smooth solution to system (1.1)-(1.2) is unique,
we get that $(u,p,\theta)\equiv(\tilde{u},\tilde{\theta},\tilde{p})$
on $(t_{0},T^{*})$. Thus, $(\tilde{u},\tilde{\theta},\tilde{p})$
is a global smooth continuation of $(u,p,\theta)$. This completes
the proof of Theorem \ref{mainthm}. \end{proof}
\begin{remark}
The above argument does not give any information about the possible growth of the solution with respect to time. But it does
	work.
	\end{remark}

The proof of Proposition \ref{proposition} is a
direct consequence of Lemma \ref{mainlemma}.
\begin{proof}[Proof
	of Proposition \ref{proposition}]
Choosing $p,q$ such that $\frac{1}{p}+\frac{1}{q}>1$ in Lemma \ref{imbedding}, we have $H^1 \cap L_{\sigma}^{2} \hookrightarrow D_{A_{q}}^{1-\frac{1}{p},p}$. Then, it follows from  Lemma \ref{mainlemma} that
 \[
 u\in L^{\infty}(0,T;H^{1})\cap L^{2}(0,T;H^{2}),\;  \nabla u\in L^1(0,T;L^\infty).
 \]
 Applying Lemma \ref{tran} yields that 
 \[\theta \in L^{\infty}(0,T;H^{1}).\] This completes the proof of Proposition \ref{proposition}.
\end{proof}
\section*{Acknowledgements}

The author is supported in part by the National Natural Science Foundation
of China (No. 11401176) and Doctor Fund of Henan Polytechnic University
(No. B2012-110).

\bibliographystyle{amsplain}

\end{document}